\documentclass[english,11pt, a4paper, final]{article}	
\usepackage[T1]{fontenc}
\usepackage[utf8]{inputenc}
\usepackage[verbose, a4paper, marginparwidth=2cm]{geometry}
\usepackage{color}
\usepackage[english]{babel}
\usepackage{booktabs}
\usepackage{amsthm, mathtools, amssymb, nicefrac}
\mathtoolsset{showonlyrefs}

\usepackage{enumitem}
\setlist[enumerate,1]{label=(\roman*), ref=(\roman*)}

\usepackage{graphicx}
\usepackage[numbers]{natbib}

\usepackage[notcite,notref, draft]{showkeys}	

\usepackage{newtxtext, newtxmath}	
\usepackage{dsfont}   
\usepackage{microtype}

\usepackage{siunitx}
\usepackage{tikz}
\usetikzlibrary{trees,matrix,arrows.meta}
\usepackage[ruled, vlined]{algorithm2e}			

\usepackage{caption, subcaption}

\usepackage[pdftitle={Sinkhorn}, pdfauthor={Alois Pichler}, citecolor=darkgray, urlcolor= darkgray, linkcolor= darkgray, unicode=true,  bookmarks=true, bookmarksnumbered=false, bookmarksopen=false, breaklinks=false, pdfborder={0 0 0}, pdfborderstyle={}, backref=page, colorlinks=true]{hyperref}
\usepackage{doi}
\usepackage[textsize= tiny, obeyFinal]{todonotes}

\theoremstyle{plain}
	\newtheorem{theorem}{Theorem}[section]		
	\newtheorem{corollary}[theorem]{Corollary}
	\newtheorem{lemma}[theorem]{Lemma}
	\newtheorem{prop}[theorem]{Proposition}
\theoremstyle{definition}
	\newtheorem{definition}[theorem]{Definition}
\theoremstyle{remark}
	\newtheorem{remark}[theorem]{Remark}

\numberwithin{equation}{section}  

\DeclareMathOperator{\E}{{\mathds E}}

\DeclareMathOperator{\diag}{{diag}}
\DeclareMathOperator{\proj}{{proj}}

\newcommand{\one}{{\mathds 1}} 		
\newcommand{\nd}{\boldsymbol d}
\newcommand{\nde}{\boldsymbol{ de}}	

\title{\textbf{Nested Sinkhorn Divergence To Compute\\ The Nested Distance}}
\author{
	Alois Pichler\thanks{ University of Technology, Chemnitz, Faculty of Mathematics. 90126 Chemnitz, Germany}\,
	\thanks{DFG, German Research Foundation~-- Project-ID 416228727 -- SFB~1410. \protect\hfill\protect\\
		\includegraphics[width=0.8em]{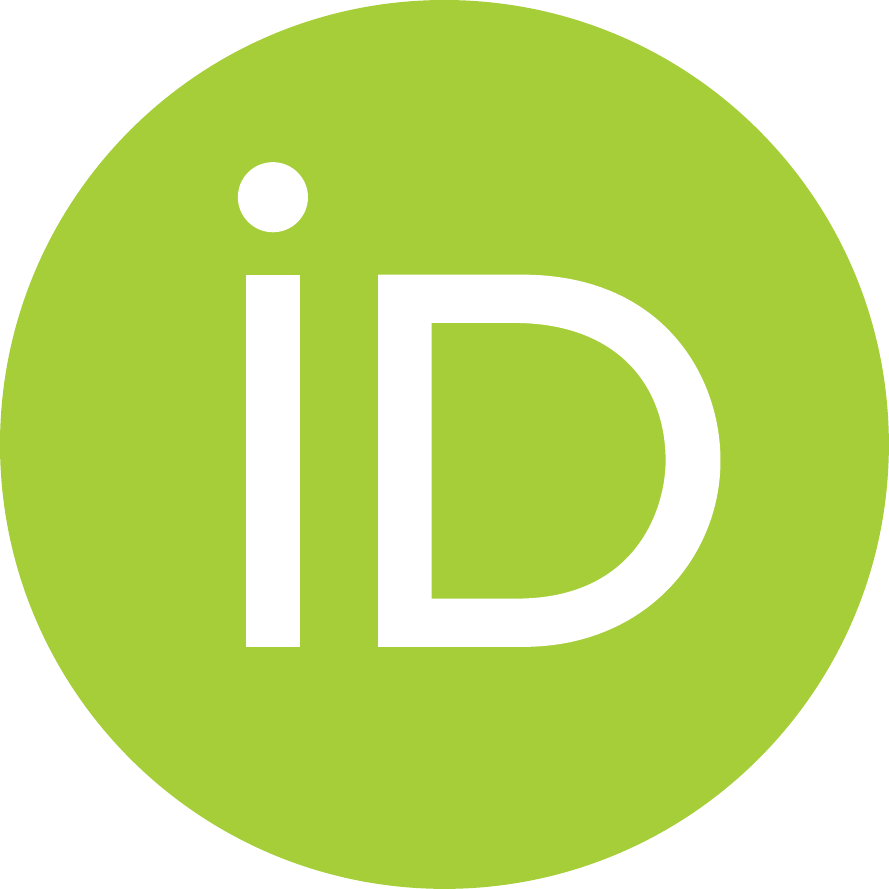}~\protect\href{https://orcid.org/0000-0001-8876-2429}{orcid.org/0000-0001-8876-2429}.
	Contact: \protect\href{mailto:alois.pichler@math.tu-chemnitz.de}{alois.pichler@math.tu-chemnitz.de}}  
\and Michael Weinhardt\footnotemark[1]}

\begin{document}
	\maketitle
	\begin{abstract}
		The nested distance builds on the Wasserstein distance to quantify the difference of stochastic processes, including also the information modelled by filtrations.
		The Sinkhorn divergence is a relaxation of the Wasserstein distance, which can be computed considerably faster. For this reason we employ the Sinkhorn divergence and take advantage of the related (fixed point) iteration algorithm. Furthermore, we investigate the transition of the entropy throughout the stages of the stochastic process and provide an entropy-regularized nested distance formulation, including a characterization of its dual.
		Numerical experiments affirm the computational advantage and supremacy.
		

		\noindent \textbf{Classification:} 90C08, 90C15, 60G07 
	\end{abstract}

\section{Introduction}

The Wasserstein distance, also known as Monge--Kantorovich distance, is used in optimal transport theory to describe and characterize optimal transitions between probability measures. They are characterized by the lowest (or cheapest) average costs to fully transfer a probability measure into another. The costs are most typically proportional to the distance of locations to be connected. \citet{RachevRueschendorf} provide a comprehensive discussion of the Wasserstein distance and \citet{Villani2009} summarizes the optimal transport theory. 

The nested distance is based on the Wasserstein distance. It has been introduced by \citet{Pflug2009} and generalizes and extends the theory from probability measures to stochastic processes, cf.\ \citet{PflugPichler2011}. 

The nested distance is employed in multistage stochastic programming to describe the quality of an approximation. Multistage stochastic programming has applications in many sectors, e.g., the financial sector (\citet{Edir05}, \citet{Brodt1983}), in management science or in energy economics (\citet{AnaluiPflug}, \citet{Welington2017, Carpentier, Carpentier2015}).
The prices, demands, etc., are often modeled as a stochastic process $\xi=(\xi_0,\dots,\xi_T)$ and the optimal values are rarely obtained analytically.
For the numerical approach the stochastic process is replaced by a finite valued stochastic scenario process $\tilde\xi=(\tilde\xi_0,\dots,\tilde\xi_T)$, which is a finite tree. Naturally, the approximation error should be minimized without unnecessarily increasing the complexity of the computational effort. \citet{KiruiPichler} provide a Julia package for generating scenario trees and scenario lattices for multistage stochastic programming. \citet{MaggioniPflug} provide guaranteed bounds and \citet{KopaVitali2020} investigate corresponding reduction techniques.

This paper addresses the \emph{Sinkhorn divergence} in place of the Wasserstein distance. This pseudo-distance is also called \emph{Sinkhorn distance} or \emph{Sinkhorn loss}. In contrast to the exact implementation \citet{Bertsekas1989}, e.g., Sinkhorn divergence corresponds to a regularization of the Wasserstein distance, which is strictly convex and which allows to improve the efficiency of the computation by applying Sinkhorn's (fixed-point) iteration procedure. The relaxation itself is similar to the modified objective of interior-point methods in numerical optimization. 
A cornerstone is the theorem by \citet{Sinkhorn1967a} that shows a unique decomposition for non-negative matrices and ensures convergence of the associated iterative scheme. \citet{Cuturi2013} has shown the potential of the Sinkhorn divergence and made it known to a wider audience. Nowadays, Sinkhorn divergence is used in statistical applications, cf.\ \citet{BigotCazelles2019} and \citet{Luise2018Differential}, for image recognition and machine learning, cf.\ \citet{Kolouri2017} and \citet{PeyreAude2018}, among many other applications.

Extending Sinkhorn's algorithm to multistage stochastic programming has been proposed recently in \citet[Section~5.2.3, pp.~97--99]{Tran2020}, where a numerical example indicating computational advantages is also given. This paper resumes this idea and assesses the entropy relaxed nested distance from theoretical perspective. We address its approximating properties and derive its convex conjugate, the dual.
As well, numerical tests included confirm the computational advantage regarding the simplicity of the implementation as well as significant gains in speed.

\paragraph{Outline of the paper.}
	The following Section~\ref{sec:Preliminaries} introduces the notation and provides the definitions to discuss the nested distance. Additionally, the importance of the filtration and the complexity of the computation is shown. Section~\ref{sec:SinkhornDistance} introduces the Sinkhorn divergence and derive its dual. In Section~\ref{sec:EntropyRegularizedNestedDistance} we regularize the nested distance and show the equality between two different approaches. Results and comparisons are visualized and discussed in Section~\ref{sec:Experiments}. Section~\ref{sec:Summary} summarizes and concludes the paper.

\section{Preliminaries}\label{sec:Preliminaries}
	This section recalls the definition of the nested distance and provides an example to illustrate the importance of the filtration. Throughout, we shall work on a probability space~\((\Xi, \mathcal F, P)\).

\subsection{Wasserstein distance}
The Wasserstein distance is a distance for probability measures. It is the building block for the process distance and its regularized version, which we address here, the Sinkhorn divergence. The Sinkhorn divergence is not a distance in itself. To point out the differences we highlight the defining elements.
\begin{definition}[Distance of measures]\label{def:Distance}
	Let~$\mathcal P$ be a set of probability measures on $\Xi$. A function
	$d\colon\mathcal P\times\mathcal P\to[0,\infty)$ 
	is called \emph{distance}, if it satisfies the following conditions: 
	\begin{enumerate}
		\item\label{enu:Nonnegativity} Nonnegativity: for all $P_1$, $P_2\in\mathcal P$,
		\[	d(P_1,P_2)\geq 0;	\]
		\item Symmetry: for all $P_1$, $P_2\in\mathcal P$, 
		\[	d(P_1,P_2)= d(P_2,P_1);\]
		\item\label{enu:Triangle} Triangle Inequality: for all $P_1$, $P_2$ and $P_{3}\in\mathcal P$,
		\[	d(P_1,P_2)\leq d(P_1,P_3)+d(P_3,P_2);\]
		\item\label{enu:Strict} Strictness: if $d(P_1,P_2)=0$, then $P_1=P_2$. 
	\end{enumerate}
\end{definition}
\citet{Rachev} presents a huge variety of probability metrics. 
Here, we focus on the Wasserstein distance, which allows a generalization for stochastic processes. For this we assume that the sample space~\(\Xi\) is equipped with a metric~\(d\) so that~\((\Xi,d)\) is Polish.
\begin{definition}[Wasserstein distance]
	Let~$P$ and~$\tilde P$ be two probability measure on $\Xi$ endowed with a distance $d\colon\Xi\times\Xi\to\mathbb R$. The \emph{Wasserstein distance} of order $r\geq1$ is
	\[
		d^r(P,\tilde P)\coloneqq\inf_\pi \iint_{\Xi\times\Xi}d(\xi,\tilde\xi)^r\,\pi(\mathrm d\xi,\mathrm d\tilde\xi),
	\]
	where the infimum is over all probability measures $\pi$ on $\Xi\times\Xi$ 	with marginals $P$ and $\tilde P$, respectively. 
\end{definition}

\begin{remark}[Distance versus cost functions]
	The definition of the Wasserstein distance presented here starts with a distance~$d$ on $\Xi$ and the Wasserstein distance is a distance on $\mathcal P$ in the sense of Definition~\ref{def:Distance} above. 
	However, in what follows \emph{any} cost function $c\colon \Xi\times\Xi\to\mathbb R$ could be considered instead of the distance~$d$ on $\Xi$ (of course,~$c$ has to be measurable and the integral has to exist). The result might not be a distance in the sense of Definition~\ref{def:Distance}. In what follows we will point to the differences.
\end{remark}

In a discrete framework, probability measures are of the form \(P=\sum\nolimits _{i=1}^np_i\,\delta_{\xi_i}\) with $p_i\geq0$ and $\sum\nolimits _{i=1}^np_i=1$ and the support $\{\xi_i\colon i= 1,2,\dots,n\}\subset\Xi$ is finite. 
The Wasserstein distance~$d^r$ of two discrete measures $P=\sum_{i=1}^np_i\,\delta_{\xi_i}$ and $\tilde P= \sum_{j=1}^{\tilde n}\tilde p_j\,\delta_{\tilde\xi_j}$ is the $r$-th root of the optimal value of 
\begin{align}\label{eq:Wasserstein}
	\text{minimize }_{\text{in }\pi\ } & \sum_{i=1}^n\sum_{j=1}^{\tilde n} \pi_{ij}\,d_{ij}^r\\
	\text{subject to } & \sum_{j=1}^{\tilde n}\pi_{ij}=p_i,&& i=1,\dots,n, \\
	& \sum_{i=1}^n\pi_{ij}=\tilde p_j, && j=1,\dots\tilde n \text{ and}\\
	&\pi_{ij}\ge 0,
\end{align}
where $d_{ij}\coloneqq d(\xi_i,\tilde\xi_j)$ is an $n\times\tilde n$-matrix collecting all distances. The optimal measure in~\eqref{eq:Wasserstein} is denoted~$\pi^W$ and called an optimal transport plan.
The convex, linear dual of~\eqref{eq:Wasserstein} is 
\begin{subequations}
	\begin{align}\label{eq:WassersteinDual}
		\text{maximize}_{\text{ in }\lambda\text{ and } \mu} & \sum_{i=1}^n p_i\,\lambda_i+\sum_{j=1}^{\tilde n} \tilde p_j\,\mu_j\\ \label{eq:WassersteinDualb}
		\text{subject to }& \lambda_i+\mu_j\le d_{ij}^r  \ \text{ for all }i=1,\dots n \text{ and }  j=1,\dots\tilde n.
	\end{align}
\end{subequations}

\begin{remark}
	The problem~\eqref{eq:Wasserstein} can be written as linear optimization problem
	\begin{align*}
		\text{minimize}_{\text{ in }x\ } & c^{\top}x\\
		\text{subject to } & Ax=b,\\
		& x\geq0,
	\end{align*}
	where $x=(\pi_{11},\pi_{21},\dots,\pi_{n\tilde n})^\top$, $c=(d_{11}, d_{21},\dots,d_{n\tilde n})^\top$, $b=(p_1,\dots,p_n,\tilde p_1,\dots,\tilde p_{\tilde n})^\top$
	and~$A$ is the matrix
	\[	A=\begin{pmatrix}\one_{\tilde n}\otimes I_n\\
		I_{\tilde n}\otimes\one_n
		\end{pmatrix}\]
		with $\one=(1,\dots,1)$.
	\end{remark}

\subsection{The distance of stochastic processes}
Let $(\Xi,\mathcal F,P)$ and $(\tilde\Xi,\tilde{\mathcal F},\tilde P)$
be two probability spaces. We now consider two stochastic processes with realizations $\xi$, $\tilde\xi\in \Xi$ and $\Xi\coloneqq \Xi_0 \times \Xi_1\times\dots\times \Xi_T$.  There are many metrics $d$ such that $(\Xi,d)$
is a metric space. Without loss of generality we may set $\Xi_t=\mathbb R$ for all $t\in\{0,1,\dots,T\}$ and employ the $\ell^1$-distance, i.e., $d(\xi,\tilde\xi)= \sum_{t=0}^T|\xi_t-\tilde\xi_t|$.

\begin{remark}
	The example depicted in Figure~\ref{fig:example} illustrates that a simple application of the Wasserstein distance does not capture the different information (knowledge) available at the intermediate stage.  Indeed, let $\epsilon>0$. 
	The distance matrix of the trajectories is 
	\[	d=\begin{pmatrix}\epsilon & 2+\epsilon\\	2 & 0
		\end{pmatrix}
	\]
	and the optimal transport plan is 
	\[ 	\pi=\frac12\begin{pmatrix}1 & 0\\
		0 & 1 \end{pmatrix}.
	\]
	It follows the Wasserstein distance according~\eqref{eq:Wasserstein} is $d=\sum_{i,j}d_{ij}\,\pi_{ij}=\nicefrac\epsilon2$.
	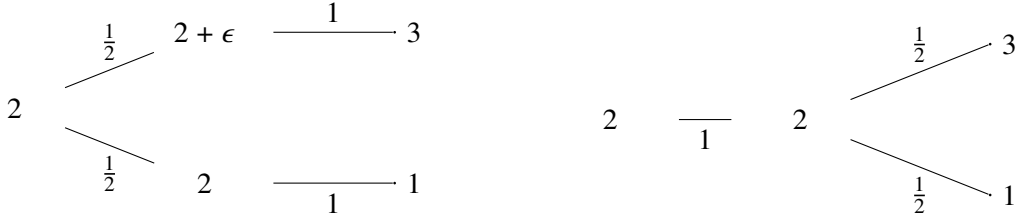
\begin{figure}
		\centering \pagestyle{empty} 
		\tikzstyle{level 1}=[level distance=2.5cm, sibling distance=2cm]
		\tikzstyle{level 2}=[level distance=2.5cm, sibling distance=1cm]
		\tikzstyle{bag} = [text width=4em, text centered] 
		\tikzstyle{end} = [circle, minimum width=0pt,fill, inner sep=0pt]

		\begin{minipage}[b]{0.5\linewidth}
		\begin{tikzpicture}[grow=right]
			\node[bag] {2}     child {         node[bag] {2}                     child {                 node[end, label=right:                     {1}] {}                 edge from parent                 node[below]  {$1$}             }             edge from parent              node[below]  {$\frac1{2}$}     }     child {         node[bag] {$2+\epsilon$}                     child {                 node[end, label=right:                     {3}] {}                 edge from parent                 node[above]  {$1$}             }         edge from parent                      node[above]  {$\frac1{2}$}     };
		\end{tikzpicture}
		\end{minipage}
		\begin{minipage}[b]{0.5\linewidth}
		\centering
		\pagestyle{empty} 
		\tikzstyle{level 1}=[level distance=2.5cm, sibling distance=2cm]
		\tikzstyle{level 2}=[level distance=2.5cm, sibling distance=2cm]
		\tikzstyle{bag} = [text width=4em, text centered]
		\tikzstyle{end} = [circle, minimum width=0pt,fill, inner sep=0pt]
		\begin{tikzpicture}[grow=right] 
		\node[bag] {2}     child {         node[bag] {2}                 child {                 node[end, label=right:                     {1}] {}                 edge from parent                 node[below]  {$\frac1{2}$}             }             child {                 node[end, label=right:                     {3}] {}                 edge from parent                 node[above]  {$\frac1{2}$}             }         edge from parent                      node[below]  {$1$}     }; 
		\end{tikzpicture}
		\end{minipage}

		\caption{\label{fig:example}Two processes illustrating two different flows
		of information, cf.\ \citet{Heitsch2006}, \citet{KovacevicPichler}}
	\end{figure}
\end{remark}

We conclude from the preceding remark that the Wasserstein distance is not suitable to distinguish stochastic processes with different flows of information. The reason is that this approach does not involve conditional probabilities at stages $t=0,1,\dots,T-1$, but only probabilities at the final stage $t=T$, where all the information from intermediate stages are ignored.
The information at the previous stage is encoded by the $\sigma$\nobreakdash-algebra 
\[\mathcal F_t=\sigma\big(A_1\times\dots\times A_t\times \Xi_{t+1}\times\dots\times \Xi_T\colon A_{t^\prime}\subset\Xi_{t^\prime} \text{ measurable}\big)\]
for $t=0,1,\dots,T$ ($\tilde{\mathcal F}_t$, resp.). 
The following generalization of the Wasserstein distance takes all conditional probabilities into account. 
\begin{definition}[The nested distance]\label{ndDefinition}
	The \emph{nested distance} of order $r\geq1$
	of two filtered probability spaces $\mathbb P=(\Xi,(\mathcal F_t),P)$
	and $\tilde{\mathbb P}=(\tilde\Xi, (\tilde{\mathcal F}_t), \tilde P)$,
	for which a distance $d\colon\Xi\times\tilde\Xi\to\mathbb R$
	is defined, is the optimal value of the optimization problem 
	\begin{align}\label{eq:NestedDistance}
		\text{minimize}_{\text{ in }\pi\ } & \left(\iint_{\Xi\times\tilde\Xi}
		d(\xi,\tilde\xi)^r\, \pi(\mathrm d\xi, \mathrm d\tilde\xi)\right)^{\nicefrac1r}\\
		\text{subject to } & \pi(A\times\tilde\Xi\mid \mathcal F_t\otimes\tilde{\mathcal F}_t)=P(A\mid \mathcal F_t), &  & A\in\mathcal F_t,\ t=1,\dots,T,\\
		& \pi(\Xi\times B\mid \mathcal F_t\otimes\tilde{\mathcal F}_t)=\tilde P(B\mid \tilde{\mathcal F}_t), &  & B\in\tilde{\mathcal F}_t,\ t=1,\dots,T,
	\end{align}
	where the infimum in~\eqref{eq:NestedDistance} is among all bivariate
	probability measures $\pi\in\mathcal P(\Xi\times\tilde\Xi)$
	defined on $\mathcal F_T\otimes\tilde{\mathcal F}_T$. The optimal value of~\eqref{eq:Wasserstein}, the nested distance of order~$r$, is denoted by~$\nd^r(\mathbb P,\tilde{\mathbb P})$. 
\end{definition}

For the discrete nested distance we use trees to model the whole space
and filtration. We denote by $\mathcal N_t$ ($\tilde{\mathcal N}_t$,
resp.)\ the set of all nodes at the stage $t$. Furthermore, a predecessor~$m$ of the node~$i$, not necessarily the immediate predecessor, is indicated by~$m\prec i$. 
The nested distance for trees is the $r$-th root of the optimal value of 
\begin{align}\label{eq:ndTree}
	\text{minimize }_{\text{in } \pi\ } & \sum_{i,j}\pi_{ij}\cdot d_{ij}^r\\
	\text{subject to } & \sum_{j\succ j_t}\pi(i,j\mid i_t,j_t)=P(i\mid i_t), && i_t\prec i,j_t,\\
	& \sum_{i\succ i_t}\pi(i,j\mid i_t,j_t)=\tilde P(j\mid j_t), && j_t\prec j,i_t,\\
	& \pi_{ij}\geq0 \text{ and } \sum_{i,j}\pi_{ij}=1,
\end{align}
where $i\in\mathcal N_T$ and $j\in\tilde{\mathcal N}_T$
are the leaf nodes and $i_t\in\mathcal N_t$ as well as $j_t\in\tilde{\mathcal N}_t$ are nodes on the same stage $t$. As usual for discrete measures, the conditional probabilities $\pi(i,j\mid i_t,j_t)$ are given by 
\begin{equation}\label{eq:piCond}
	\pi(i,j\mid i_t,j_t)\coloneqq\frac{\pi_{ij}}{\sum_{i'\succ i_t,j'\succ j_t}\pi_{i^\prime j^\prime}}.
\end{equation}

\begin{remark}
	Employing the definition~\eqref{eq:piCond} for $\pi(i,j\mid i_t,j_t)$ reveals that the problem~\eqref{eq:ndTree} is indeed a \emph{linear} program in~$\pi$ (cf.~\eqref{eq:Wasserstein}).
\end{remark}

\subsection{Rapid, nested computation of the process distance}

This subsection addresses an advanced approach for solving the linear
program~\eqref{eq:ndTree}. We first recall the tower property, which allows an important simplification of the constraints in~\eqref{eq:NestedDistance}.
\begin{lemma}\label{TowerProperty} 
	To compute the nested distance it is enough to condition on the immediately following $\sigma$\nobreakdash-algebra: the conditions
	\[	\pi\big(A\times\Xi\mid \mathcal F_t\otimes\tilde{\mathcal F}_t\big)\ 
		\text{ for all }\ A\in\mathcal F_T\]
	in~\eqref{eq:NestedDistance} may be replaced by 
	\[	\pi\big(A\times\Xi\mid\mathcal F_t\otimes\tilde{\mathcal F}_t\big)\ 
	\text{ for all }\ A\in\mathcal F_{t+1}.\]
 \end{lemma}

\begin{proof}
	The proof is based on the tower property of the expectation and can be found in \citep[Lemma~2.43]{PflugPichlerBuch}. 
\end{proof}
As a result of the tower property the full problem~\eqref{eq:ndTree} can be calculated faster in a recursive way and the matrix for the constraints has not to be stored. We employ this result in an algorithm blow. For further details we refer to \citet[Chapter~2.10.3]{PflugPichlerBuch}.
The collection of all direct successors of node $i_t$ ($j_t$, resp.)\ is denoted by $i_t+$ ($j_t+$, resp.).

\begin{algorithm}[t]
	\KwIn{for all combinations of leaf nodes $i\in\mathcal N_T$ and $j\in\tilde{\mathcal N}_T$ with predecessors $(i_0,i_1,\dots,i_{T-1},i)$ and $(j_0,j_1,\dots,j_{T-1},j)$ set $\nd_T^r(i,j):= d\left((\xi_0,\xi_{i_1},\dots,\xi_i),\ (\tilde\xi_0,\tilde\xi_{j_1},\dots,\tilde\xi_j)\right)^r$}
	\KwOut{the optimal transport plan at the leaf nodes $i\in\mathcal N_T$ and $j\in\tilde{\mathcal N}_T$ is $\pi(i,j)=\pi_1(i_1,j_1\mid i_0,j_0)\cdot\dots\cdot\pi_{T-1}(i,j\mid i_{T-1},j_{T-1})$.}
	\SetKwFor{Loop}{for}{do}{end}
	\Loop{$t=T-1$ down to $0$ and every combination of inner nodes $i'\in\mathcal N_t$ and $j'\in\tilde{\mathcal N}_t$}{
		\text{solve the linear programs}
	\begin{align}
		\text{minimize}_{\text{ in }\pi\ }  & \sum_{i'\in i_t+,\, j' \in j_t+} \pi(i',j'\mid i_t,j_t)\cdot \nd_{t+1}^r(i',j')\label{eq:20}\\
	\text{subject to } & \sum_{j'\in j_t+} \pi(i',j'\mid i_t,j_t)=P(i'\mid i_t) , &i'\in i_t+,\\ 
		&\sum_{i'\in i_t+} \pi(i',j'\mid i_t,j_t)=\tilde P(j'\mid j_t), &j'\in j_t+, \\
		&\pi(i',j'\mid i_t,j_t)\ge 0
	\end{align}
	and denote its optimal value by $\nd_t^r(i_t,j_t)$.}
	\caption{Nested computation of the nested distance $\nd^r(\mathbb{P},\tilde{\mathbb P})$ of two tree-processes~\(\mathbb P\) and~\(\tilde{\mathbb P}\)}
	\KwResult{The nested distance is $\nd^r(\mathbb P,\tilde{\mathbb P})\coloneqq \nd^r_0(0,0)$}
	\label{alg:ndTreeRecursive}
\end{algorithm}

\section{Sinkhorn divergence\label{sec:SinkhornDistance}}

In what follows we consider the entropy-regularization of the Wasserstein
distance~\eqref{eq:Wasserstein} and characterize its dual. Moreover,
we recall Sinkhorn's algorithm, which allows and provides a considerably faster implementation. These results are combined then to accelerate the computation of the nested distance. 


\subsection{Entropy-regularized Wasserstein distance}

Interior point methods add a logarithmic penalty to the objective to force the optimal solution of the modified problem into the strict interior. The Sinkhorn distance proceeds similarly. The regularizing term $H(x)\coloneqq-\sum_{i,j}x_{ij}\log x_{ij}$ is added to the cost function in problem~\eqref{eq:Wasserstein}. This has shown beneficiary in other problem settings as well.
\begin{remark}\label{rem:phi}
	The mapping $\varphi(x)\coloneqq x\log x$ is convex and negative for $x\in(0,1)$ with continuous extensions $\varphi(0)=\varphi(1)=0$ so that $H\ge 0$, provided that all $x_{ij}\in[0,1]$.
\end{remark}
\begin{definition}[Sinkhorn divergence]\label{DefSinkhornDistance}
	The \emph{Sinkhorn divergence} is obtained by the optimization problem 
	\begin{subequations}
		\begin{align}\label{eq:Sinkhorn}
			\text{minimize}_{\text{ in } \pi\ } & { \sum_{i,j}\pi_{ij}\, d_{ij}^r-\frac1\lambda H(\pi)}\\
			\text{subject to } & \sum_j\pi_{ij}=p_i, &  & i=1,\dots,n,\\
			& \sum_i\pi_{ij}=\tilde p_j, &  & j=1,\dots,\tilde n,\\
			&  \pi_{ij}>0 &  & \text{ for all } i, j,\label{eq:Sinkhornb}
		\end{align}
	\end{subequations}
	where $d$ is a distance or a cost matrix and $\lambda>0$ is a regularization parameter. 
	With $\pi^S$ being the optimal transport in~\eqref{eq:Sinkhorn}--\eqref{eq:Sinkhornb} we denote the Sinkhorn divergence by
	\begin{equation*}
		d_S^r\coloneqq\sum_{i,j}\pi_{ij}^S\, d_{ij}^r
	\end{equation*}
	and the Sinkhorn divergence including the entropy by 
	\begin{equation*}
		de_S^r\coloneqq\sum_{i,j}\,\pi_{ij}^S\,d_{ij}^r-\frac1\lambda H\big(\pi^S\big).
	\end{equation*}
\end{definition}

We may mention here that we avoid the term Sinkhorn \emph{distance}
since for all $\lambda>0$ the Sinkhorn divergence $d_S^r$
is strictly positive and $de_S^r$ can be negative
for small $\lambda$ which violates the axioms of a distance given in Definition~\ref{def:Distance} above (particularly~\ref{enu:Nonnegativity},~\ref{enu:Triangle} and~\ref{enu:Strict}).
Strict positivity of~$d_S^r$ can be forced by a correction term, the so-called Sinkhorn Loss (see \citet[Definition 2.3]{BigotCazelles2019}) or by employing the cost matrix $d\cdot\one_{p\neq\tilde p}$ instead.
\begin{remark}
	The strict inequality constraint~\eqref{eq:Sinkhornb} is not a restriction. Indeed, the mapping $\varphi(x)$ defined in Remark~\ref{rem:phi} has derivative $\varphi'(0)=-\infty$ and thus it follows that every optimal measure satisfies the strict inequality~$\pi_{ij}>0$ for $\lambda>0$.
\end{remark}

We have the following inequalities.
\begin{prop}[Comparison of Sinkhorn and Wasserstein]\label{prop:Inequality1} 
	It holds that 
	\begin{equation}\label{eq:16}
		de_S^r\leq d_W^r\leq d_S^r.
	\end{equation}
\end{prop}

\begin{proof}
	Recall that $\pi \log\pi\le0$ for all $\pi\le1$ and  thus it holds that
	$\sum_{i,j}\pi_{ij}\,d_{ij}^r+\frac1\lambda\sum_{i,j}\pi_{ij}\log\pi_{ij}\leq\sum_{i,j}\pi_{ij}\,d_{ij}^r$
	for all $\pi\in(0,1]^{n\times\tilde n}$. It follows that
	\[	\min_{\pi}\ \sum_{i,j}\pi_{ij}\,d_{ij}^r +\frac1\lambda\sum_{i,j}\pi_{ij}\log\pi_{ij}
	 	\leq\min_\pi\ \sum_{i,j}\pi_{ij}\,d_{ij}^r\]
	and thus the first inequality. The remaining inequality is clear by the definition of the Wasserstein distance.
\end{proof}
Both Sinkhorn divergences $d_S^r$ and $de_S^r$ approximate the Wasserstein distance $d_W^r$, and we have convergence for $\lambda\to\infty$ to $d_W^r$. The following proposition provides precise bounds.
\begin{prop}\label{prop:SinkhornInequality}
	For every $\lambda>0$ we have 
	\begin{equation}\label{eq:13}
		0\le d_S^r-d_W^r\le \frac1\lambda \left( H(\pi^S)-H(\pi^W)\right)
	\end{equation} 
	and
	\begin{equation}\label{eq:15}		
		0\le d_W^r - de_S^r \le \frac1\lambda H(\pi^S)\leq\frac1\lambda H(p\cdot\tilde p^{\top})
	\end{equation}
	with $p=(p_1,\dots,p_n)$ and $\tilde p=(\tilde p_1,\dots,\tilde p_{\tilde n})$,
	respectively. 
\end{prop}
\begin{proof}
	The first inequalities follow from~\eqref{eq:16} and from optimality of~$\pi^S$ in the inequality
	\begin{equation*}
		d_S^r-\frac1\lambda H(\pi^S)\le d_W^r-\frac1\lambda H(\pi^W).
	\end{equation*}
	The latter again with~\eqref{eq:16} and $d_S^r-de_S^r= \frac1\lambda H\big(\pi^S\big)$. Finally, by the log sum inequality, $H(\pi)\le H\big(p\cdot \tilde p^\top\big)$ for every measure $\pi$ with marginals $p$ and $\tilde p$.
\end{proof}
\begin{remark}\label{rem:3}
	As a consequence of the log sum inequality we obtain as well that $H(\pi^S)\le \log n+\log\tilde n$. The inequalities~\eqref{eq:13} and ~\eqref{eq:15} thus give strict upper bounds in comparing the Wasserstein distance and the Sinkhorn divergence.
\end{remark}

\paragraph{Alternative definitions.}
There exist alternative definitions of the Sinkhorn divergence which
we want to mention here. The first alternative definition involves
the Kullback\textendash Leibler divergence $D_\textit{KL}(\pi\mid P\otimes\tilde P)$,
which is defined as
\[	D_\textit{KL}(\pi\mid P\otimes\tilde P)\coloneqq-\sum_{i,j}\pi_{ij}\log\frac{\pi_{ij}}{p_i\,\tilde p_j}=H(P)+H(\tilde P)-H(\pi),\]
where the latter equality is justified provided that $\pi$ has marginal measures~$P$ and~$\tilde P$.
The Sinkhorn divergence (in the alternative definition) is the $r$\nobreakdash-th root of the optimal value of
\begin{subequations}
	\begin{align}\label{eq:KLDivergence}
	\text{minimize}_{\text{ in }\pi} & \sum_{i,j}\pi_{ij}\,d_{ij}^r\\
		\text{subject to } & \sum_j\pi_{ij}=p_i, &  & i=1,\dots,n,\\
		& \sum_i\pi_{ij}=\tilde p_j, &  & j=1,\dots,\tilde n,\\
		& \pi_{ij}>0 &&\text{and}\\
		& D_\textit{KL}(\pi\mid P\otimes\tilde P)\leq\alpha & & \text{for all }i,j, \label{eq:KL}
	\end{align}
\end{subequations}
where $\alpha\geq0$ is the regularization parameter. For each $\alpha$
in~\eqref{eq:KL} we have by the duality theory a corresponding $\lambda$ in~\eqref{eq:Sinkhorn} such that the optimal values coincide.
Let $\alpha>0$ and $\pi^\textit{KL}$ be the solution to problem~\eqref{eq:KLDivergence}--\eqref{eq:KL} with Lagrange multipliers $\beta$ and $\gamma$. Then the optimal value of problem~\eqref{eq:KLDivergence} equals $d_S^r$ from~\eqref{eq:Sinkhorn} with 
\[
	\lambda=-\frac{\log(\pi^\textit{KL}_{ij})+1}{d_{ij}+\beta_i+\gamma_j}
\]
for any $i\in\{1,\dots,n\}$ and $j\in\{1,\dots,\tilde n\}$. For
further information and illustration we refer to \citet[Section~3]{Cuturi2013}. 
\medskip

A further, potential definition employs a different entropy regularization is given by
\[\tilde H(\pi)=-\sum_{i,j}\pi_{ij}\cdot(\log\pi_{ij}-1).\]
\citet{Luise2018Differential} use this definition for Sinkhorn approximation
for learning with Wasserstein distance and proof an exponential convergence.
This definition leads to a similar matrix decomposition and iterative
algorithm described in the following sections.

\subsection{Dual representation of Sinkhorn}

We shall derive Sinkhorn's algorithm and its extension to the nested distance via duality.  To this end consider the Lagrangian function 
\begin{equation}\label{eq:Lagrange}
	L(\pi;\beta,\gamma)\coloneqq 
	\sum_{i,j}\pi_{ij}\,d_{ij}+\frac1\lambda \sum_{i,j}\pi_{ij}\log\pi_{ij} +\beta^\top(p-\pi\cdot\one)+(\tilde p-\one^{\top}\cdot\pi)^\top\gamma
\end{equation}
of the problem~\eqref{DefSinkhornDistance}.
The partial derivatives are
\begin{equation}\label{eq:PartialDerivative}
	\frac{\partial L}{\partial\pi_{ij}}=\frac1\lambda \left(\log\pi_{ij}+1\right)+d_{ij}-\beta_i-\gamma_j=0,
\end{equation}
and it follows from~\eqref{eq:PartialDerivative} that the optimal measure has entries
\begin{align}\label{eq:piRepresentation}
	\pi_{ij}^* & =\exp\left(-\lambda(d_{ij}-\beta_i-\gamma_j)-1\right)\\
	& =\diag\Big(\exp(\lambda\,\beta-\nicefrac12)\Big)\cdot\exp(-\lambda\,  d)\cdot\diag\Big(\exp(\lambda\,\gamma-\nicefrac12)\Big).
\end{align}
By inserting $\pi_{ij}^*$ in the Lagrangian function~$L$ we get the convex dual function
\begin{align*}
	\MoveEqLeft[2] d(\beta,\gamma) \coloneqq\inf_{\pi}L(\pi;\beta,\gamma)=L(\pi^{*};\beta,\gamma)\\
 	& =\sum_{i,j}d_{ij}\cdot e^{-\lambda(d_{ij}-\beta_i-\gamma_j)-1}-\frac1\lambda \sum_{i,j}e^{-\lambda(d_{ij}-\beta_i-\gamma_j)-1}\cdot\big(\lambda(d_{ij}-\beta_i-\gamma_j)+1\big)\\
	 & \qquad+\sum_i\beta_i\left(p_i-\sum_je^{-\lambda(d_{ij}-\beta_i-\gamma_j)-1}\right)+\sum_j\gamma_j\left(\tilde p_j-\sum_ie^{-\lambda(d_{ij}-\beta_i-\gamma_j)-1}\right)\\
 	& =-\sum_{i,j}\left(\beta_i+\gamma_j+\frac1\lambda \right)e^{-\lambda(d_{ij}-\beta_i-\gamma_j)-1}+\sum_i\beta_ip_i+\sum_j\gamma_j\tilde p_j\\
 	& \qquad-\sum_i\beta_i\left(\sum_je^{-\lambda(d_{ij}-\beta_i-\gamma_j)-1}\right)-\sum_j\gamma_j\left(\sum_ie^{-\lambda(d_{ij}-\beta_i-\gamma_j)-1}\right)\\
 	& =\sum_i\beta_i\,p_i+\sum_j\gamma_j\,\tilde p_j-\frac1\lambda \sum_{i,j}e^{-\lambda(d_{ij}-\beta_i-\gamma_j)-1}.
\end{align*}
The dual problem thus is
\begin{align*}
	\text{maximize}_{\text{ in }\beta, \gamma} &  \sum_i\beta_i\,p_i+\sum_j\gamma_j\,\tilde p_j-\frac1\lambda \sum_{i,j}e^{-\lambda(d_{ij}-\beta_i-\gamma_j)-1}\\
	\text{subject to } & \beta\in\mathbb R^n,\ \gamma\in\mathbb R^{\tilde n}.
\end{align*}
Due to $\sum_{i,j}e^{-\lambda(d_{ij}-\beta_i-\gamma_j)-1}=1$
we may write the latter problem as 
\begin{subequations}	
	\begin{align}\label{eq:14}
		\text{maximize}_{\text{ in }\beta,\gamma\ } & \sum_i\beta_i\,p_i+\sum_j\gamma_j\,\tilde p_j\\
		\text{subject to } & \sum_{i,j}e^{-\lambda(d_{ij}-\beta_i-\gamma_j)-1}=1\text{ and } \beta\in\mathbb R^n,\ \gamma\in\mathbb R^{\tilde n}.\label{eq:14b}
	\end{align}
\end{subequations}

\begin{remark}
	We deduce from~\eqref{eq:14b} that $-\lambda \left(d_{ij}- \beta_i- \gamma_j\right)- 1\leq0$, or
	\begin{equation}\label{eq:17}
		\beta_i+\gamma_j\le d_{ij}+\frac1\lambda \quad\text{ for all }i, j
	\end{equation}
	provided that~$\lambda>0$.
	It is thus apparent that~\eqref{eq:14}--\eqref{eq:14b} is a relaxation of problem~\eqref{eq:WassersteinDual}--\eqref{eq:WassersteinDualb} together  with the constraint~\eqref{eq:17}.
	As well, observe that both problems coincide for $\lambda\to\infty$ in~\eqref{eq:14b}.
\end{remark}

\subsection{Sinkhorn's algorithm\label{SinkhornAlgorithm}}
To derive Sinkhorn's algorithm we consider the Lagrangian function~\eqref{eq:Lagrange} again, but now for the remaining variables.  Similar to $\pi^*$ in~\eqref{eq:piRepresentation}, the gradients are
\begin{equation}\label{eq:derivitiveBeta}
	\frac{\partial L}{\partial\beta_i}=p_i- \sum_{j=1}^{\tilde n}\pi_{ij}= p_i-\sum_{j=1}^{\tilde n}e^{-\lambda(d_{ij}-\beta_i-\gamma_j)-1}=0
\end{equation}
and 
\begin{equation}\label{eq:derivitiveGamma}
	\frac{\partial L}{\partial\gamma_j}= \tilde p_j-\sum_{i=1}^n\pi_{ij} =\tilde p_j-\sum_{i=1}^ne^{-\lambda(d_{ij}-\beta_i-\gamma_j)-1}=0
\end{equation}
so that the equations
\begin{equation*}
	\beta_i=\frac1\lambda \log\left(\frac{p_i}{\sum_{j=1}^{\tilde n}e^{-\lambda(d_{ij}-\gamma_j)-1}}\right)
	\quad\text{and}\quad
	\gamma_j=\frac1\lambda \log\left(\frac{\tilde p_j}{\sum_{i=1}^ne^{-\lambda(d_{ij}-\beta_i)-1}}\right)
\end{equation*}
follow.
To avoid the logarithm introduce $\tilde \beta_i\coloneqq e^{\lambda\,\beta_i-\nicefrac12}$ and $\gamma_j\coloneqq e^{\lambda\,\gamma_j-\nicefrac12}$ and rewrite the latter equations as
\begin{equation}\label{eq:Sinkhorn1}
	\tilde\beta_i= \frac{p_i}{\sum_{j=1}^{\tilde n}e^{-\lambda\,d_{ij}}\,\tilde \gamma_j} \quad\text{and}\quad
	\tilde\gamma_j= \frac{\tilde p_j}{\sum_{i=1}^n \tilde \beta_i\,e^{-\lambda\,d_{ij}}},
\end{equation}
while the optimal transition plan~\eqref{eq:piRepresentation} is
\[\pi_{ij}^*= \tilde\beta_i\cdot e^{-\lambda\,d_{ij}} \cdot \tilde \gamma_j.\]

\begin{algorithm}[t]
	\KwIn{distance matrix $d^r\in\mathbb R_{\geq0}^{n\times\tilde n}$, probability vectors $p\in\mathbb R_{\geq0}^n$, $\tilde p\in\mathbb R_{\geq0}^{\tilde n}$, regularization parameter $\lambda>0$, stopping criterion and a starting value $\tilde\gamma=(\tilde\gamma_1,\dots,\tilde\gamma_{\tilde n})$}
	\KwOut{$\tilde\beta$, $\tilde\gamma$ for $\diag(\tilde\beta)\cdot e^{-\lambda\, d^r}\cdot \diag(\tilde\gamma)$}
	 set
	 \begin{equation}\label{eq:MatrixK}
		k_{ij}\coloneqq \exp\big(-\lambda\,d_{ij}^r\big).
	 \end{equation}

	 \While{stopping criterion is not satisfied}{
	  \For{$i=1$ \KwTo $n$}{
	  $\tilde\beta_i\leftarrow\frac{p_i}{\sum_{j=1}^{\tilde n}k_{ij}\,\tilde\gamma_j}$
		}
		\For{$j=1$ \KwTo $\tilde n$}{
		$\tilde\gamma_j\leftarrow\frac{\tilde p_j}{\sum_{i=1}^n\tilde\beta_i\,k_{ij}}$}
	 }
	\KwResult{The matrix $\pi_{ij}^*=\tilde\beta_i \, e^{-\lambda\,d_{ij}^r}\,\tilde\gamma_j=\tilde\beta_i\,k_{ij}\,\tilde\gamma_j$ solves the relaxed Wasserstein problem~\eqref{eq:Sinkhorn}--\eqref{eq:Sinkhornb}.}
	\caption{Sinkhorn's iteration}\label{alg:SinkhornIteration}
\end{algorithm}

The simple starting point of Sinkhorn's iteration is that~\eqref{eq:Sinkhorn1} can be used to determine~$\tilde\beta$ and~$\tilde\gamma$ alternately.
Indeed, from~\eqref{eq:derivitiveBeta} and~\eqref{eq:derivitiveGamma}
we infer that~$\pi^*$ is a doubly stochastic matrix and Sinkhorn's theorem (cf.\ \citet{Sinkhorn1967a, Sinkhorn1967}) for the matrix decomposition ensures that iterating~\eqref{eq:Sinkhorn1} converges and the vectors $\tilde\beta$ and $\tilde\gamma$ are unique up to a scalar.
Algorithm~\ref{alg:SinkhornIteration} summarizes the individual steps again.
\begin{remark}[Central path]
	We want to emphasize that for changing the regularization parameter~$\lambda$ it is note necessary to recompute all powers in~\eqref{eq:MatrixK}. Indeed, increasing $\lambda$ to $2\cdot\lambda$, for example, corresponds to raising all entries in the matrix~\eqref{eq:MatrixK} to the power~$2$, etc.
\end{remark}
\begin{remark}[Softmax]
	The expression~\eqref{eq:Sinkhorn1} resembles to what is known as the \emph{Gibbs measure} and to the \emph{softmax} in data science. 
\end{remark}
\begin{remark}[Historical remark]
	In the literature, this approach is also known as \emph{matrix scaling} (cf.\ \citet{RoteZachariasen}), 
	RAS (cf.\ \citet{Bachem}) as well as Iterative Proportional Fitting
	(cf.\ \citet{Ruschendorf1995}). \citet{Kruithof1937} used the method
	for the first time in telephone forecasting. 
	The importance of this iteration scheme for data science was probably observed in \citet[Algorithm~1]{Cuturi2013} for the first time. 
\end{remark}

\section{Entropic transitions\label{sec:EntropyRegularizedNestedDistance}}
This section extends the preceding sections and combines the Sinkhorn divergence and the nested distance by incorporating the regularized entropy $\frac1\lambda H(\pi)$ to the recursive nested distance Algorithm~\ref{alg:ndTreeRecursive} and investigate its properties and consequences. We characterize the nested Sinkhorn divergence first. The main result is used to exploit duality.

\subsection{Nested Sinkhorn divergence}
Let $\nde^{(t)}$ be the matrix of incremental divergences of sub-trees at stage~$t$.
Analogously to~\eqref{eq:20} we consider the conditional version of the problem~\eqref{eq:Sinkhorn} and denote by $\beta_{i_tj_t}$ and $\gamma_{j_ti_t}$ the pair of optimal Lagrange parameters associated with the problem 
\begin{align}\label{program:SinkhornRecursive}
	\text{minimize}_{\text{ in} \pi} & \sum_{i'\in i_t+,j'\in j_t+}\pi(i',j'\mid i_t,j_t)\cdot\nde^{(t+1)}(i',j')\\
	& \hspace{4cm}+\frac1\lambda\pi(i',j'\mid i_t,j_t)\cdot\log\pi(i',j'\mid i_t,j_t)\\
	\text{subject to } & \sum_{j'\in j_t+}\pi(i',j'\mid i_t,j_t)=P(i'\mid i_t),\hspace{2cm}i'\in i_t+,\\
	& \sum_{i'\in i_t+}\pi(i',j'\mid i_t,j_t)=\tilde P(j'\mid j_t),\hspace{2cm}j'\in j_t+,\\
	& \pi(i',j'\mid i_t,j_t)>0,
\end{align}
where $\pi(i',j'|i_t,j_t)=\exp \left(-\lambda\big(\nde_{i_tj_t}^{(t+1)} -\beta_{i_tj_t}- \gamma_{j_ti_t}\big)-1\right)$.
The optimal value is the new divergence $\nde^{(t)}(i_t,j_t)$.
Computing the nested distance recursively from $t=T-1$ down to~$0$ we get 
\begin{align}\label{eq:dRecursiveSinkhorn}
	\pi_{ij} & =\pi_1(i_1,j_1\mid i_0,j_0)\cdot\ldots\cdot\pi_{T-1}(i,j\mid i_{T-1},j_{T-1})\\
	& =e^{-\lambda(\nde_{i_0j_0}^{(1)}-\beta_{i_0j_0}-\gamma_{j_0i_0})-1}\cdot\ldots\cdot e^{-\lambda(\nde_{i_{T-1}j_{T-1}}^{(T)}-\beta_{i_{T-1}j_{T-1}}-\gamma_{j_{T-1}i_{T-1}})-1}\\
	& =\exp\left(-T-\lambda\sum_{t=0}^{T-1}\nde_{i_tj_t}^{(t+1)}-\beta_{i_tj_t}-\gamma_{j_ti_t}\right),
\end{align}
where $i\in\mathcal N_T$ and $j\in\tilde{\mathcal N}_T$ are the leaf nodes with predecessors $(i_0,i_1,\dots,i_{T-1},i)$ and $(j_0,j_1,\dots,j_{T-1},j)$.
As above introduce
\[\tilde\beta_{i_tj_t}\coloneqq \exp{\big(\lambda\,\beta_{i_tj_j}-\nicefrac12\big)} \quad\text{and}\quad
\tilde\gamma_{j_ti_t}\coloneqq\exp{\big(\lambda\,\gamma_{j_ti_t}-\nicefrac12\big)}.\] 
Combining the components it follows that
\begin{align}\label{dRecurSum}
	\pi_{ij} & =\exp\left(-T-\lambda\sum_{t=0}^{T-1} \nde_{i_tj_t}^{(t+1)}-\beta_{i_tj_t}-\gamma_{j_ti_t}\right)\\
	& =\prod_{t=0}^{T-1}\tilde\beta_{i_tj_t}\exp\left(-\lambda\, \nde_{i_tj_t}^{(t+1)}\right)\, \tilde\gamma_{j_ti_t},
\end{align}
where the product is the entry-wise product (Hadamard product).

The following theorem summarizes the relation of the nested distance with the Sinkhorn divergence.

\begin{theorem}[Entropic relaxation of the nested distance]\label{thm:NestedSinkhorn}
	The recursive solution~\eqref{program:SinkhornRecursive} (\eqref{eq:dRecursiveSinkhorn}, resp.)\ coincides with the optimal transport plan given by
	\begin{align}\label{eq:SinkhornFull}
		\text{minimize}_{\text{ in }\pi} & \sum_{i,j}\pi_{ij}\cdot d_{ij}^r+\frac1\lambda\pi_{ij}\cdot\log\big(\pi_{ij}\big)\\
		\text{subject to } & \sum_{j\succ j_t+}\pi(i,j\mid i_t,j_t)=P(i\mid i_t), & i_t\prec i,j_t,\\
		& \sum_{i\succ i_t+}\pi(i,j\mid i_t,j_t)=\tilde P(j\mid j_t), & j_t\prec j,i_t,\\
		& \pi_{ij}>0\,\,\text{and}\,\,\sum_{i,j}\pi_{ij}=1.
	\end{align}
\end{theorem}

\begin{proof}
	First define $\pi\coloneqq\prod_{t=1}^T\pi_t$, where~$\pi_t$ is the conditional transition probability, i.e., the solution at stage~$t$ and the matrices are multiplied element-wise (the Hadamard product) as in equation~\eqref{eq:dRecursiveSinkhorn} above. It follows that
	\begin{align}\label{eq:21}
		d^r\cdot\pi+\frac1\lambda\pi\log\pi & =d^r\cdot\prod_{t=1}^T\pi_t+\frac1\lambda\cdot \prod_{t=1}^T\pi_t\log\left(\prod_{t=1}^T\pi_{t}\right)\\
		& =d^r\cdot\prod_{t=1}^T\pi_t+\frac1\lambda\cdot \prod_{t=1}^T\pi_t \cdot \sum_{t=1}^T\log\pi_t.
	\end{align}
	Observe that $\pi_t(A)=\mathbb E(1_A\mid\mathcal F_t\otimes\tilde{\mathcal F}_t)$
	(cf.\ Lemma~\eqref{TowerProperty}). Denote the $r$\nobreakdash-distance of subtrees by~$\nde_T^r$. By linearity of the conditional expectation we have with~\eqref{eq:21}
	\[	\nde_{T-1}	=\E\left[\nde_T^r+\frac1\lambda\log\pi_T\Big{|}\, \mathcal F_{T-1}\otimes\tilde{\mathcal F}_{T-1}\right]^{\nicefrac1r}\]
	and from calculation in backward recursive way 
	\begin{align*}
		\MoveEqLeft[1] \nde_{T-2} =\E\left[\nde_{T-1}^r+\frac1\lambda\log\pi_{T-1}\Big{|}\, \mathcal F_{T-2}\otimes\tilde{\mathcal F}_{T-2}\right]^{\nicefrac1r}\\
		& =\E\left[\E\left[\nde_T^r +\frac1\lambda\log\pi_T\Big{|}\mathcal F_{T-1}\otimes\tilde{\mathcal F}_{T-1}\right] +\frac1\lambda\log\pi_{T-1}\Big{|}\, \mathcal F_{T-2}\otimes\tilde{\mathcal F}_{T-2}\right]^{\nicefrac1r}\\
		& =\E\left[\E\left[\nde_T^r +\frac1\lambda\log\pi_T +\frac1\lambda\log\pi_{T-1}\Big{|}\mathcal F_{T-1}\otimes\tilde{\mathcal F}_{T-1} \right]\Big{|}\mathcal F_{T-2}\otimes\tilde{\mathcal F}_{T-2}\right]^{\nicefrac1r}.
	\end{align*}
	Finally, it follows that
	\begin{align*}
		\MoveEqLeft[1]\nde_0 =\mathbb E\left[\nde_1^r+\frac1\lambda\log\pi_1\Big{|}\mathcal F_0\otimes\tilde{\mathcal F}_0\right]^{\nicefrac1r}\\
		& =\E\left[\E\left[\dots\E\left[\nde_T^r+\frac1\lambda\log\pi_T\Big{|}\mathcal F_{T-1}\otimes\tilde{\mathcal F}_{T-1}\right]\dots\Big{|}\mathcal F_1\otimes\tilde{\mathcal F}_1\right] +\frac1\lambda\log\pi_1\Big{|}\mathcal F_0\otimes\tilde{\mathcal F}_0\right]^{\nicefrac1r}\\
		& =\E\left[\E\left[\dots\E\left[\nde_T^r +\frac1\lambda\sum_{t=1}^T\log\pi_t\Big{|}\mathcal F_{T-1}\otimes\tilde{\mathcal F}_{T-1}\right]\dots\Big{|}\mathcal F_1\otimes\tilde{\mathcal F}_1\right]\Big{|}\mathcal F_0\otimes\tilde{\mathcal F}_0\right]^{\nicefrac1r}\\
		& =\E\left[\nde_T^r +\frac1\lambda\sum_{t=1}^T\log\pi_t\Big{|}\mathcal F_0\otimes\tilde{\mathcal F}_0\right]^{\nicefrac1r}\\
		& =\E\left[\nde_T^r +\frac1\lambda\sum_{t=1}^T\log\pi_t\right]^{\nicefrac1r},
	\end{align*}
	the assertion~\eqref{eq:SinkhornFull} of the theorem. 
\end{proof}
\begin{remark}
	The optimization problem in Theorem~\ref{thm:NestedSinkhorn} considers all constraints as the full nested problem~\eqref{eq:ndTree}, only the objective differs. For this reason the optimal solution of~\eqref{eq:SinkhornFull} is feasible for the problem~\eqref{eq:ndTree} and vice versa.

	Notice as well that the tower property can be used in a forward calculation. 
\end{remark}

Similarly to Proposition~\ref{prop:SinkhornInequality} we have the following extension to the nested Sinkhorn divergence.
\begin{corollary}\label{cor:Inequality} 
	For the nested distance and the nested Sinkhorn divergence, the same inequalities as in Proposition~\ref{prop:SinkhornInequality} apply, i.e.,
	\[
		0\le \boldsymbol d_S^r -\boldsymbol d_W ^r \le \frac1\lambda \left(H(\pi^S)-H(\pi^W)\right)
		\quad\text{and}\quad
		0\le \boldsymbol d_W^r-\boldsymbol{de}_S^r\le\frac1\lambda H(\pi^S)\le\frac1\lambda H(p\cdot p^\top),
	\]
	where $\pi^S$ ($\pi^W$, resp.)\ is the optimal transport plan from~\eqref{eq:SinkhornFull} (\eqref{eq:ndTree}, resp.)\ with discrete, unconditional probabilities~$p$ and~$\tilde p$ at the final stage~$T$. 
\end{corollary}
\begin{proof}
	The proof follows the lines of the proof of the Propositions~\ref{prop:Inequality1} and~\ref{prop:SinkhornInequality}.
\end{proof}

Moreover, we have the following general inequality that allows an error bound depending on the total~$T$ of stages.

\begin{corollary}
	Let $m$ ($\tilde m$, resp.)\ be the maximum number of immediate successors in the process~$\mathbb{P}$ ($\tilde{\mathbb{P}}$, resp.), i.e., $m=\max\left\{ |i+|\colon i\in\mathcal N_t,\ t=1,\dots,T-1\right\} $.
	It holds that
	\begin{equation}\label{eq:28}
		\boldsymbol{de}_S^r-\boldsymbol d_W^r\leq\frac{\log m+\log\tilde m}\lambda\cdot T,
	\end{equation}
	where~$T$ is the total number of stages.
\end{corollary}

\begin{proof}
	Recall from Remark~\ref{rem:3} that $H(\pi^S)\leq\log(n\,\tilde n)= \log n+\log\tilde n$ for every conditional probability measures, where~$n$ and~$\tilde n$ are the number of immediate successors in both trees. The result follows with $n\le m^T$ ($\tilde n\le \tilde m^T$, resp.)\ and $\log n\le T\log m$ and the nested program~\eqref{program:SinkhornRecursive}.
\end{proof}

\subsection{Nested Sinkhorn duality}
The nested distance is of importance in stochastic optimization because of its dual, which is characterized by the Kantorovich--Rubinstein theorem, cf.~\eqref{eq:WassersteinDual}--\eqref{eq:WassersteinDualb} above. The nested distance allows for a characterization by duality as well. Here we develop the duality for the nested Sinkhorn divergence. In line with Theorem~\ref{thm:NestedSinkhorn} we need to consider the problem
\begin{subequations}
	\begin{align}\label{ndRegularized}
		\text{minimize}_{\text{ in }\pi} & \left(\iint \left( d(\xi,\tilde\xi)^r+\frac1\lambda\log\pi(\xi, \tilde \xi)\right) 
		\pi(\mathrm d\xi,\mathrm d\tilde\xi)\right)^{\nicefrac1r}\\
		\text{subject to } & \pi(A\times\tilde\Xi\mid \mathcal F_t\otimes\tilde{\mathcal F}_t)=P(A\mid \mathcal F_t), \qquad A\in\mathcal F_t,\ t=1,\dots,T, \label{ndRegularizedb}\\
		& \pi(\Xi\times B\mid \mathcal F_t\otimes\tilde{\mathcal F}_t)=\tilde P(B\mid \tilde{\mathcal F}_t), \qquad B\in\tilde{\mathcal F}_t,\ t=1,\dots,T. \label{ndRegularizedc}
	\end{align}		
\end{subequations}
However, we first reformulate the problem~\eqref{eq:14}--\eqref{eq:14b}.
By translating the dual variables,
$\hat\beta\coloneqq -\beta+\E \beta$ 
and $\hat\gamma\coloneqq -\gamma+\tilde{\E}\gamma$,
and defining $M_0\coloneqq-\E\beta-\tilde\E\gamma$
we have the alternative representation 
\begin{align}\label{M0Sinkhorn}
	\text{maximize}_{\text{ in }M_0\ } & M_0\\
	\text{subject to } & \E\hat\beta=0,\ \tilde\E\hat\gamma=0,\\
	& \sum_{\xi,\tilde\xi}\exp\left(-\lambda\left(d(\xi,\tilde\xi)^r-\hat\beta(\xi)-\hat{\gamma}(\tilde\xi)-M_0\right)-1\right)=1,\\
	& \hat\beta\in\mathbb R^n,\ \hat\gamma\in\mathbb R^{\tilde n}.
\end{align}
To establish the dual representation of the nested distance we introduce the projections 
\begin{align*}
	\proj_t\colon L^1(\mathcal F_T\otimes\tilde{\mathcal F}_T) & \to L^1(\mathcal F_t\otimes\tilde{\mathcal F}_T)\\
	\hat{\beta}(\xi)\cdot\hat{\gamma}(\tilde\xi) & \mapsto\E(\hat{\beta}\mid \mathcal F_t)(\xi)\cdot\hat{\gamma}(\tilde\xi)
\end{align*}
and 
\begin{align*}
	\tilde{\proj}_t\colon L^1(\mathcal F_T\otimes\tilde{\mathcal F}_T) & \to L^1(\mathcal F_T\otimes\tilde{\mathcal F}_t)\\
	\hat\beta(\xi)\cdot\hat{\gamma}(\tilde\xi) & \mapsto\hat{\beta}(\xi)\cdot\E(\hat{\gamma}\mid \tilde{\mathcal F}_t)(\tilde\xi).
\end{align*}

We recall the following characterization of the measurability constraints~\eqref{ndRegularizedb}--\eqref{ndRegularizedc} and refer to \citep[Proposition~2.48]{PflugPichlerBuch} for its proof.
\begin{prop}\label{Prop:CharacProj}
	The measure $\pi$ satisfies the marginal condition 
	\[
		\pi(A\times\tilde\Xi\mid \mathcal F_t\otimes\tilde{\mathcal F}_t)=P(A\mid \mathcal F_t)\quad\text{for all }A\in\Xi
	\]
	if and only if 
	\[	\E_{\pi}\beta=\E_{\pi}\proj_t\beta
		\quad\emph{for all}\quad \beta\lhd\mathcal F_T\otimes\tilde{\mathcal F}_T.\]
	Moreover, $\proj_t(\beta)=\E_{\pi}(\beta\mid \mathcal F_t\otimes\tilde{\mathcal F}_T)$
	if $\pi$ has marginal $P$. 
\end{prop}

\begin{theorem}\label{ND_MSink}
	The infimum or the nested distance including the entropy $\nde^r(\mathbb{P},\tilde{\mathbb P})$ of problem~\eqref{eq:SinkhornFull} equals the supremum of all numbers $M_0$ such that
	\[
		e^{-\lambda(d(\xi,\tilde\xi)^r-M_T(\xi,\tilde\xi))-1}\in\mathcal P(\Xi\times\tilde\Xi), \qquad(\xi, \tilde\xi)\in \Xi\times\tilde\Xi,
	\]
	where $\mathcal P(\Xi\times\tilde\Xi)$ is a set of probability measures on $(\Xi\times\tilde\Xi)$ and $M_t$ is an $\mathbb R$-valued process on $\Xi\times\tilde\Xi$ of the form
	\begin{equation}\label{eq:Mt}
		M_t=M_0+\sum_{s=1}^t\hat\beta_s+\hat\gamma_s
	\end{equation}
	and the measurable functions $\hat{\beta}_t\lhd\mathcal F_t\otimes\mathcal{\tilde F}_{t-1}$
	and $\hat{\gamma}_t\lhd\mathcal F_{t-1}\otimes\mathcal{\tilde F}_t$
	satisfy $\proj_{t-1}(\hat{\beta}_t)=0$ and $\tilde{\proj}_{t-1}(\hat{\gamma}_t)=0$. 
\end{theorem}

\begin{proof}
	With Proposition~\ref{Prop:CharacProj} rewrite the dual problem as 
	\begin{align*}
		\MoveEqLeft[4]\inf_{\pi>0}\sup_{M_0,f_t,g_t}\E_{\pi}\left[d^r+\frac1\lambda\log\pi\right]  +M_0\cdot(1-\E_{\pi}\one)+\\
		& -\sum_{s=0}^{T-1}\big(\E_{\pi}f_{s+1}-\E_{\pi}\proj_s(f_{s+1})\big)
		-\sum_{s=0}^{T-1}\big(\E_{\pi}g_{s+1}-\E_{\pi}\tilde{\proj}_s(g_{s+1})\big),
	\end{align*}
	where the second line encodes the measurability constraints.
	By the minmax theorem (cf.\ \citet{Sion}) this is equivalent to 
	\begin{align*}
		\sup_{M_0,f_t,g_t}M_0+\inf_{\pi>0}\E_{\pi}\Big[ & d^r+\frac1\lambda\log\pi-M_0\cdot\one\\
		& -\sum_{s=0}^{T-1}(f_{s+1}-\proj_s(f_{s+1}))
		 -\sum_{s=0}^{T-1}(g_{s+1}-\tilde{\proj}_s(g_{s+1}))\Big].
	\end{align*}
	The integral exists and the minimum is obtained by a probability measure 
	\[
		\pi=\exp\left(-\lambda\left(d^r-\sum_{s=0}^{T-1}(f_{s+1}-\proj_s(f_{s+1}))-\sum_{s=0}^{T-1}(g_{s+1}-\tilde{\proj}_s(g_{s+1})-M_0\right)-1\right).
	\]
	Set $\hat\beta_s\coloneqq f_s-\proj_{s-1}(f_s)$ and
	$\hat\gamma_s\coloneqq g_s-\tilde{\proj}_{s-1}(g_s)$.
	Consequently, the problem reads
	\begin{align*}
		\text{maximize}_{\text{ in }M_0\ } & M_0\\
		\text{subject to } & \exp\left[-\lambda\left(d^r-\sum_{s=1}^T\hat{\beta}_s-\sum_{s=1}^T\hat{\gamma}_s-M_0\right)-1\right]\in\mathcal P(\Xi\times\tilde\Xi)\\
		& \proj_{t-1}(\hat\beta_t)=0,\tilde{\proj}_{t-1}(\hat\gamma_t)=0,
	\end{align*}
	and thus the assertion. 
\end{proof}

The following corollary links the optimal probability measure and the stochastic process~\eqref{eq:Mt} for the optimal components $\hat\beta$ and $\hat\gamma$. 
\begin{corollary}
	The process $M_t$ in~\eqref{eq:Mt}, for which the supremum is attained, is a martingale with respect to the optimal measure~$\pi$.
\end{corollary}

\begin{proof}
	The proof of \citep[Theorem~2.49]{PflugPichlerBuch} applies with minor adaptions only.
\end{proof}

\section{Numerical results\label{sec:Experiments}}
\todo[inline, size= normalsize]{Wasserstein: $d^r$, nested: $\boldsymbol d_W^r$, Sinkhorn: $d_S^r$, nested Sinkhorn: $\boldsymbol d_S^r$ (e: mit entropie)}

\begin{figure}[ht]
	\centering \subfloat[Nested distance (blue) and Sinkhorn divergence (green, red) for regularization
	parameter $\lambda\in\{0.5,1,2,\dots,30\}$\label{fig:ComputationResultHeight3a}]{\includegraphics[trim=0 0 0 10, clip, width=0.45\textwidth]{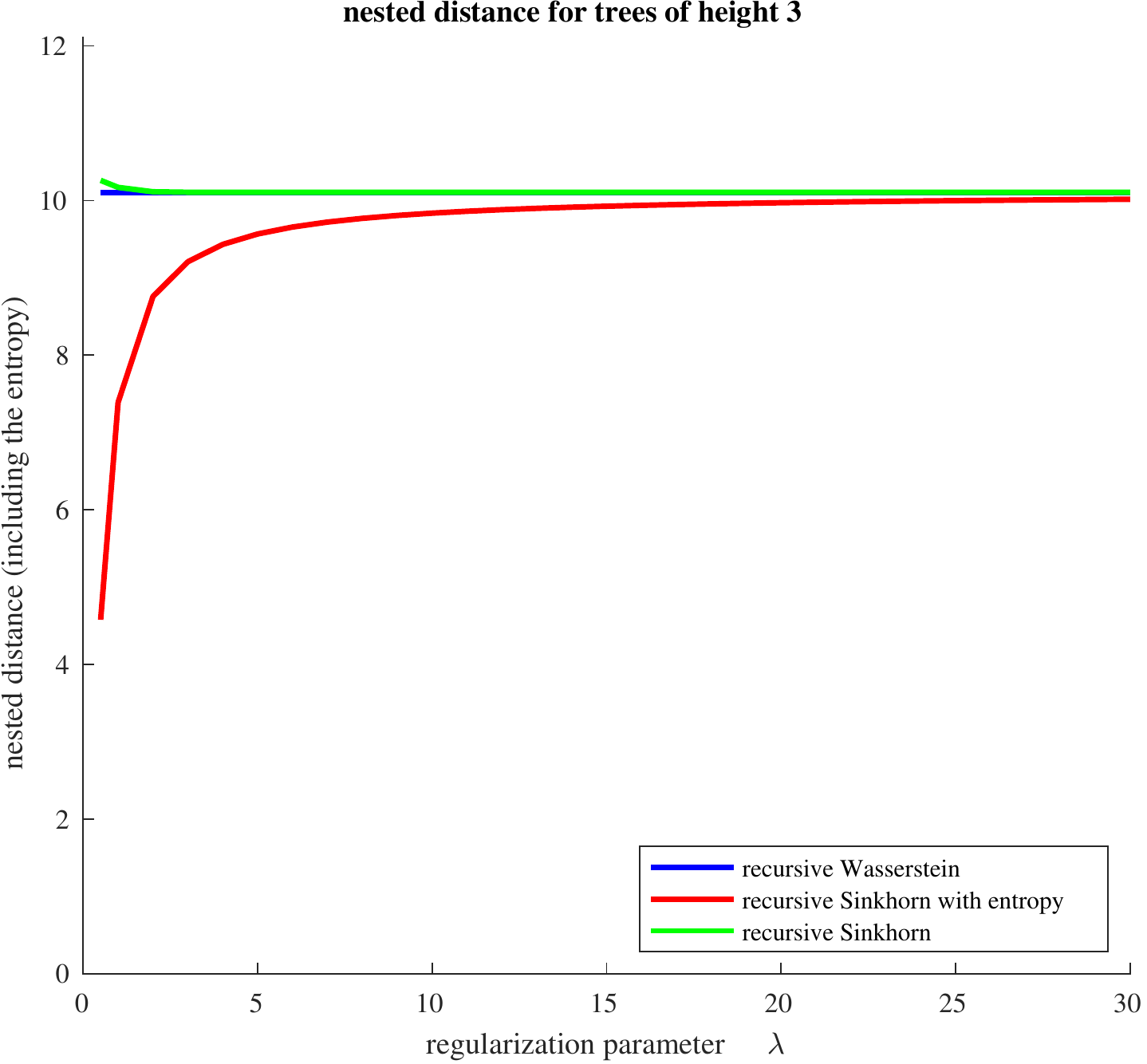} }
	\qquad\subfloat[computation time required for Algorithm~\ref{alg:ndTreeRecursive}
	(blue) and Algorithm~\ref{alg:SinkhornIteration} (green, red)\label{fig:ComputationResultHeight3b}]{\includegraphics[trim=0 0 0 10, clip, width=0.45\textwidth]{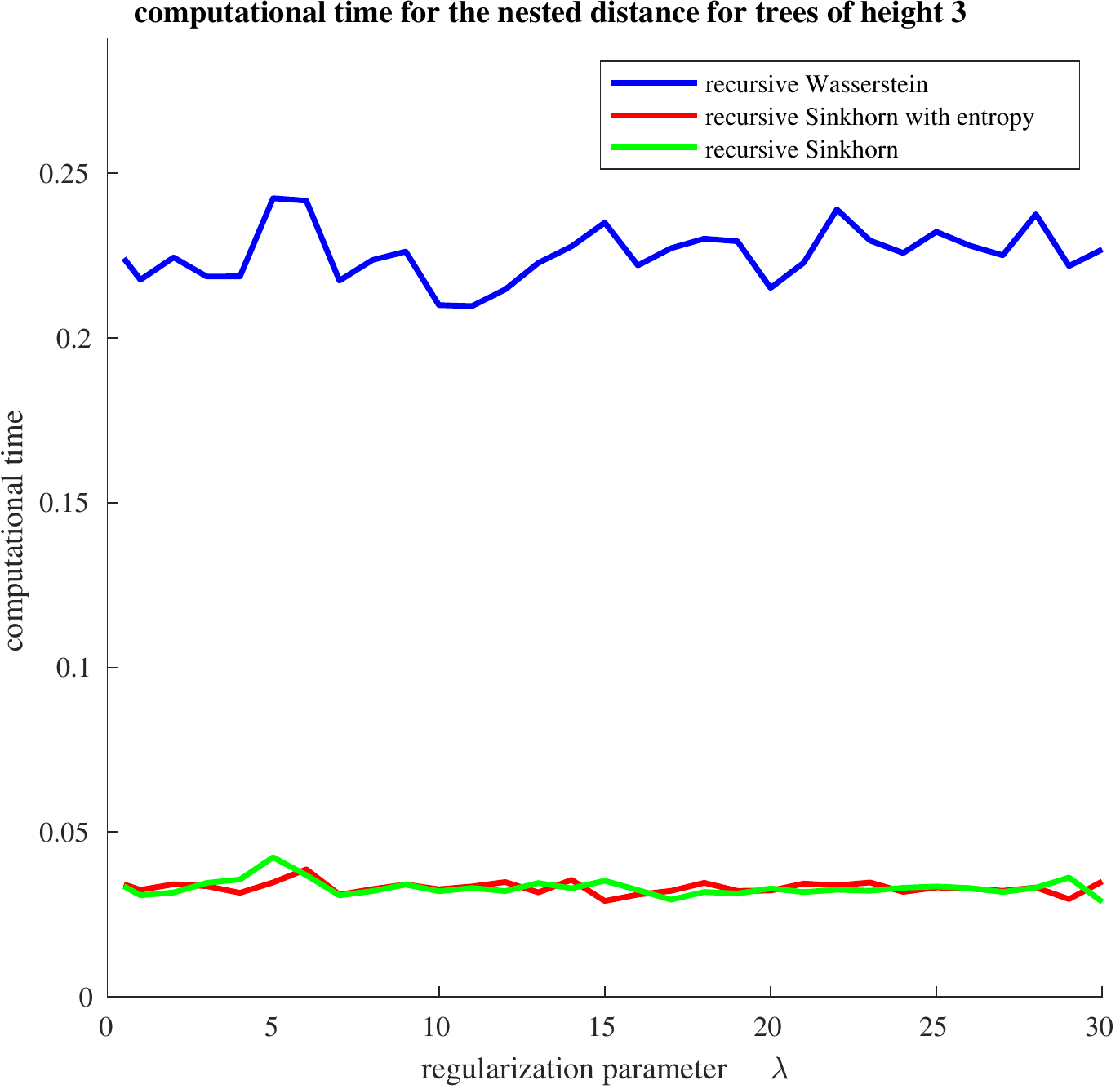} 

	}\caption{Results from computation of an arbitrary chosen processes given in
	Figure~\ref{fig:TreeComputation} with $r=1$ and $d(\xi_i,\tilde\xi_j)=|\xi_i-\tilde\xi_j|$.\label{fig:ComputationResultHeight3}}
\end{figure}
The nested Sinkhorn divergence $\boldsymbol d_S^r$ as well as $\boldsymbol {de}_S^r$ depend on the regularization parameter~$\lambda$. We discuss this dependency, the error, speed of convergence and numerical issues in comparison to the non-regularized nested distance $\nd_W^r$.

We compare Algorithm~\ref{alg:ndTreeRecursive} and  Algorithm~\ref{alg:SinkhornIteration} with respect to the nested distance $\nd_W^r$ and the nested Sinkhorn divergence with and without the entropy $\frac1\lambda  H(\pi^S)$ as well as the required computational time for two finite valued stochastic scenario processes visualized in Figure~\ref{fig:TreeComputation}.

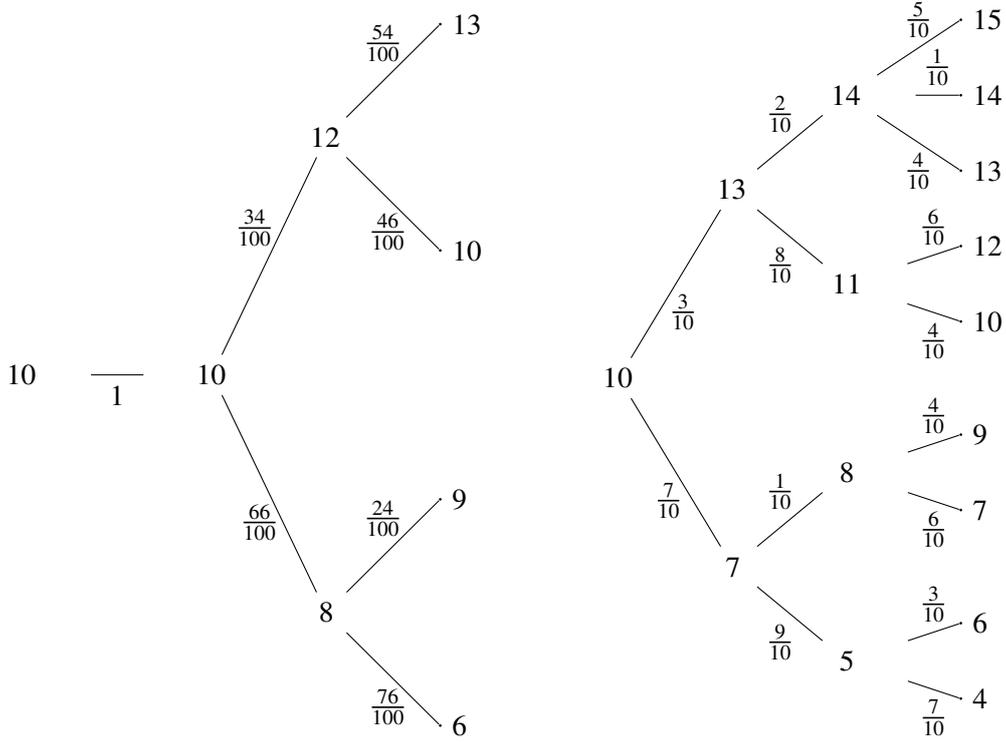
\begin{figure}[ht]
\begin{minipage}[b]{0.5\linewidth}%
\centering
\pagestyle{empty} 
\tikzstyle{level 1}=[level distance=2.5cm, sibling distance=1cm] 
\tikzstyle{level 2}=[level distance=1.5cm, sibling distance=6.3cm] 
\tikzstyle{level 3}=[level distance=1.5cm, sibling distance=3cm]
\tikzstyle{bag} = [text width=4em, text centered]
\tikzstyle{end} = [circle, minimum width=0pt,fill, inner sep=0pt] 

\begin{tikzpicture}[grow=right]
\node[bag] {10}     child {         node[bag] {10}                     child {         node[bag] {8}                     child {                 node[end, label=right:                     {6}] {}                 edge from parent                 node[below]  {$\frac{76}{100}\,\,$}             }             child {                 node[end, label=right:                     {9}] {}                 edge from parent                 node[above]  {$\frac{24}{100}\,\,\,\,$}             }             edge from parent              node[below]  {$\frac{66}{100}\,\,\,\,$}             }             child {         node[bag] {12}                     child {                 node[end, label=right:                     {10}] {}                 edge from parent                 node[below]  {$\frac{46}{100}\,\,$}             }             child {                 node[end, label=right:                     {13}] {}                 edge from parent                 node[above]  {$\frac{54}{100}\,\,\,\,$}             }             edge from parent              node[above]  {$\frac{34}{100}\,\,\,\,\,\,$}             }             edge from parent                      node[below]  {$1$}     }; 
\end{tikzpicture}%
\end{minipage}%
\begin{minipage}[b]{0.5\linewidth}%
\centering
\pagestyle{empty} 
\tikzstyle{level 1}=[level distance=1.5cm, sibling distance=5cm] 
\tikzstyle{level 2}=[level distance=1.5cm, sibling distance=2.5cm] 
\tikzstyle{level 3}=[level distance=1.5cm, sibling distance=1cm]
\tikzstyle{bag} = [text width=4em, text centered]
\tikzstyle{end} = [circle, minimum width=0pt,fill, inner sep=0pt]

\begin{tikzpicture}[grow=right]
\node[bag] {10}     child {         node[bag] {7}                     child {         node[bag] {5}                     child {                 node[end, label=right:                     {4}] {}                 edge from parent                 node[below]  {$\frac{7}{10}$}             }             child {                 node[end, label=right:                     {6}] {}                 edge from parent                 node[above]  {$\frac{3}{10}$}             }             edge from parent              node[below]  {$\frac{9}{10}\,\,\,\,$}             }             child {         node[bag] {8}                     child {                 node[end, label=right:                     {7}] {}                 edge from parent                 node[below]  {$\frac{6}{10}$}             }             child {                 node[end, label=right:                     {9}] {}                 edge from parent                 node[above]  {$\frac{4}{10}$}             }             edge from parent              node[above]  {$\frac1{10}\,\,\,\,$}             }             edge from parent                      node[below]  {$\frac{7}{10}\,\,\,$}     }     child {         node[bag] {13}                     child {         node[bag] {11}                     child {                 node[end, label=right:                     {10}] {}                 edge from parent                 node[below]  {$\frac{4}{10}$}             }             child {                 node[end, label=right:                     {12}] {}                 edge from parent                 node[above]  {$\frac{6}{10}$}             }             edge from parent              node[below]  {$\frac{8}{10}\,\,\,\,$}             }             child {         node[bag] {14}                     child {                 node[end, label=right:                     {13}] {}                 edge from parent                 node[below]  {$\frac{4}{10}$}             }             child {                 node[end, label=right:                     {14}] {}                 edge from parent                 node[above]  {$\frac1{10}$}             }             child {                 node[end, label=right:                     {15}] {}                 edge from parent                 node[above]  {$\frac{5}{10}$}             }             edge from parent              node[above]  {$\frac2{10}\,\,\,\,$}             }             edge from parent                      node[below]  {$\,\,\,\frac{3}{10}$}     };
\end{tikzpicture}%
\end{minipage}\caption{Two arbitrary chosen processes with height $T=3$.}
\label{fig:TreeComputation}
\end{figure}
 
Figure~\ref{fig:ComputationResultHeight3} displays the results.
We see that the regularized nested distance~$\nd_S^r$ (green)
and~$\boldsymbol {de}_S^r$ (red) converge to the nested distance~$\nd_W^r$ for increasing $\lambda$.
In contrast to $\nd_S^r$, the regularized nested distance including the entropy converges slower to $\nd_W^r$.
The reason is that for larger $\lambda$ the weight of the entropy
in the cost function in~\eqref{eq:Sinkhorn} decreases and the entropy of $\pi^S$ and $\pi^W$ coincide (cf.~\eqref{eq:28}).
Computing the distances with Sinkhorn's algorithm in recursive way, in contrast to solving the linear problem for the Wasserstein distance, is about six times faster. In addition, the required time for the regularized nested distance with and without the entropy varies much less by contrast with the computational time for the nested distance. 
Furthermore, the differences between $\nd_W^r$ and $\nd_S^r$ and~$\nde_S^r$, respectively, is rapidly decreasing and insignificant for $\lambda>20$.
Moreover, the time displayed in Figure~\ref{fig:ComputationResultHeight3b} does not
depend on the regularization parameter~$\lambda$.

We now fix $\lambda=20$ and vary the stages $T\in\{1,2,3,4,5\}$. The first finite tree has the branching structure $[1\ 2\ 3\ 2\ 3\ 4]$ and the second tree has a simpler structure $[1\ 2\ 2\ 1\ 3\ 2]$ (i.e., the first tree has 144 leaf nodes and the second tree 24). All states
and probabilities in the trees are generated randomly.

\begin{table}[ht]
\centering %
\begin{tabular}{cccccccc}
\toprule 
stages & \multicolumn{2}{c}{Wasserstein} & \multicolumn{3}{c}{Sinkhorn} & \multicolumn{1}{c}{difference} & \multicolumn{1}{c}{time}\tabularnewline
\midrule 
T & $\nd_W^r$ & time & $\nd_S^r$ & $\nde_S^r$ & time & $\nd_W^r-\nde_S^r$ & acceleration\tabularnewline
\midrule 
$1$ & 1.8 & 0.06\,s & 1.81 & 1.75 & 0.006\,s & 0.06 & 10$\times$\tabularnewline
$2$ & 5.1 & 0.13\,s & 5.12 & 4.97 & 0.022\,s & 0.14 & 5.8$\times$\tabularnewline
$3$ & 5.8 & 0.50\,s & 5.81 & 5.66 & 0.062\,s & 0.15 & 8.1$\times$\tabularnewline
$4$ & 7.3 & 1.54\,s & 7.32 & 7.08 & 0.368\,s & 0.24 & 4.2$\times$\tabularnewline
$5$ & 10.1 & 10.29\,s & 10.05 & 9.72 & 2.873\,s & 0.35 & 3.6$\times$\tabularnewline
\bottomrule
\end{tabular}\caption{Average distance and divergence with corresponding computational time in seconds on i5-3210M CPU. All states and probabilities are generated randomly. The regularization parameter is $\lambda=20$ and $r=1$.}
\label{tab:WassersteinSinkhornComparison}
\end{table}

Table~\ref{tab:WassersteinSinkhornComparison} summarizes the results collected.
We notice that the Sinkhorn algorithm is up to $10$~times faster compared with the usual Wasserstein distance, although the speed advantage decreases for larger trees.
The Sinkhorn algorithm also leads to small errors which increase marginally for trees with more stages. 

Additionally, we tried to improve the speed by modifying the recursive algorithm. Instead of computing once from $T-1$ down to $0$ we computed from $T-1$ down to $0$ several times to achieve a convergence in the optimal transport plan $\pi^S$. This approach has no advantages.

\section{Summary\label{sec:Summary}}
Nested distance allows distinguishing trees by involving the information encoded in filtrations.
In this paper we regularize the Wasserstein distance and introduce the Sinkhorn divergence to the nested distance. The tower property also applies for the regularization. We show that the nested divergence converges to the nested distance for increasing regularization parameter $\lambda\to\infty$. 

In conclusion, we can summarize that the Sinkhorn divergence offers a good trade-off between the regularization error and the speed advantage. Further work should focus on defining a (nested) distance for neuronal networks and extending the implementation of Sinkhorn divergence in the Julia package for faster tree generation and computation.

\section{Acknowledgement}
We are thankful to Beno\^{i}t Tran for pointing out further references, particularly on convergence, in his thesis \cite{Tran2020} supervised by Marianne Akian and Jean-Philippe Chancelier.

\bibliographystyle{abbrvnat}
\bibliography{../../Literatur/LiteraturAlois}

\end{document}